\documentclass[11pt]{amsart}
\usepackage{amscd,amssymb, amsmath, wasysym}
\usepackage{graphicx}
\usepackage{amsfonts}
\usepackage{mathrsfs}    
\usepackage{amsmath}    
\usepackage{amsthm}     
\usepackage{amscd}      
\usepackage{amssymb}    
\usepackage{eucal}      
\usepackage{latexsym}   
\usepackage{graphicx}   
\usepackage{verbatim}   
\usepackage[all]{xy}     

\makeatletter

\newcounter{thmcounter}

\numberwithin{thmcounter}{section}
\numberwithin{equation}{thmcounter}

\newtheorem{theorem}[thmcounter]{Theorem}
\newtheorem{proposition}[thmcounter]{Proposition}

\newtheorem{corollary}[thmcounter]{Corollary}

\theoremstyle{definition}
\newtheorem{definition}[thmcounter]{Definition}

\newtheorem{example}[thmcounter]{Example}

\newtheorem{remark}[thmcounter]{Remark}

\newtheorem{conjecture}[thmcounter]{Conjecture}

\newtheoremstyle{claim}{9pt}{3pt}{}{\parindent}{\bf}{.}{1em}{}

\theoremstyle{claim}
\newtheorem{claim}[equation]{Claim}

\newenvironment{namelist}[1]{%
\begin{list}{}
{
\settowidth{\labelwidth}{#1}%
\setlength{\labelsep}{0.3em}%
\setlength{\leftmargin}{\labelwidth}%
\addtolength{\leftmargin}{\labelsep}}}{%
\end{list}}

\newcommand{\nQ}{\mathbb{Q}}                     
\newcommand{\nP}{\mathbb{P}}                     

\newcommand{\nA}{\mathbb{A}}                     

\newcommand{\sO}{\mathscr{O}}                    

\newcommand{\sI}{\mathscr{I}}                    

\newcommand{\mf}[1]{\mathfrak{#1}}

\DeclareMathOperator{\Bl}{Bl}                    
\DeclareMathOperator{\bl}{Bl}                    

\DeclareMathOperator{\codim}{codim}              

\DeclareMathOperator{\exc}{exc}            

\DeclareMathOperator{\Jac}{Jac}                  

\DeclareMathOperator{\lct}{lct}                  

\DeclareMathOperator{\mld}{mld}                  
\DeclareMathOperator{\mj}{MJ}                    

\DeclareMathOperator{\ord}{ord}                  

\DeclareMathOperator{\Spec}{Spec}                
\DeclareMathOperator{\spec}{Spec}                

\newcounter{rkcounter}             
\newcommand{\MJ}[1]{{#1}_{\mj}}       

\begin{document}

\title[Mather-Jacobian singularities under generic linkage]{Mather-Jacobian singularities under generic linkage}
\author{Wenbo Niu}

\address{Department of Mathematical Sciences, University of Arkansas, Fayetteville, AR 72701, USA}
\email{wenboniu@uark.edu}



\begin{abstract} In this paper, we prove that Mather-Jacobian (MJ) singularities are preserved under the process of generic linkage. More precisely, let $X$ be a variety with MJ-canonical (resp. MJ-log canonical) singularities. Then a generic link of $X$ is also MJ-canonical (resp. MJ-log canonical). This further leads us to a result on minimal log discrepancies under generic linkage.
\end{abstract}

\maketitle

\section{Introduction}
\noindent  Two varieties $X$ and $Y$ in a nonsingular ambient variety $A$ are linked if their union $V=X\cup Y$ is a complete intersection in $A$. The study of linkage (also called liaison) has a long history in the classical algebraic geometry. Its modern study has attracted considerable attentions in the past fifty years from many authors (cf. \cite{PeskineSzpiro:LiaisonI}, \cite{LazarsfeldRao:Linkage}, \cite{Migliore:LiaisonTheory}). 

In many cases, the variety $X$ is fixed and the linked variety $Y$ is chosen to be general in the following way: choose the general equations from the defining equations of $X$ to create a complete intersection $V$ having $X$ as a component, then the residue part of $X$ in $V$ is the desired variety $Y$. In this approach, $Y$ is also called a general link of $X$.  Generic linkage theory, which follows a rigorous algebraic approach to construct general links, has been developed in a series of works mainly by Huneke and Ulrich (\cite{HunekeUlrich:DivClass}, \cite{HunekeUlrich:SturLinkage}, \cite{HU88}, \cite{HunekeUlrich:AlgLinkage}) in the past thirty years. One of the central problem in this theory is to explore and understand which properties of $X$ can be preserved by a generic link $Y$. 

Motivated by the work of Chardin-Ulrich \cite{CU:Reg}, who showed that a generic link of a local complete intersection with rational singularities has rational singularities, we are interested in studying the behavior of various singularities under generic linkage, especially for the singularities with geometric nature. In the recent work \cite{Niu:SingLink}, it has shown that log canonical singularities of pairs are always preserved under generic linkage but rational singularities are not. When the study moves to the singularities (such as canonical singularities) raised from birational geometry, the major obstruction is that the conditions of $\nQ$-Gorensteiness and normality are rarely preserved under linkage. Instead, the theory of Mather-Jacobian (MJ) singularities, rooted in the study of jet schemes, dose not require the normality condition and therefore has adequate flexibility for applications, especially for linkage. This new notion of singularities, with different names, was  introduced independently in \cite{Ishii:MatherDis}, \cite{Roi:JDiscrepancy}, \cite{Ein:MultIdeaMatherDis} and \cite{Ein:SingMJDiscrapency}.

The main purpose of this paper is to investigate MJ-singularities under the frame of linkage theory. We show that MJ-singularities are indeed preserved under the process of generic linkage.

\begin{theorem}\label{mainthm} Let $X$ be a variety with MJ-canonical (resp. MJ-log canonial) singularities. Then a generic link $Y$ of $X$ is also MJ-canonical (resp. MJ-log canonical).
\end{theorem}
The theorem is proved in section 3 (Theorem \ref{thm:31} and Theorem \ref{thm:21}) for the algebraic setting (Definition \ref{def:01}). The geometric settings (Definition \ref{def:22} and \ref{def:21}) are discussed in Section 4. Here we would like to mention that for MJ-log canonical singularities, the theorem is an immediate consequence of the results of \cite{Niu:SingLink} by using the Inversion of Adjunction (\cite{Ishii:MatherDis}, \cite{Roi:JDiscrepancy}). The new part is the case of MJ-canonical singularities, to which the most part of the paper is devoted. As a quick application of the theorem, we recover the aforementioned result of Chardin-Ulrich. Indeed, a variety that is a local complete intersection with rational singularities is MJ-canonical and MJ-canonical singularities imply rational singularities. 

One of the crucial points to establish Theorem \ref{mainthm} is to analyze the intersection divisor $Z=X\cap Y$ through a suitable resolution of singularities and use the Inversion of Adjunction. It turns out that the irreducible components of $Z$ are the minimal log canonical centers of the pair $(A, cV)$.  Hence by the Subadjunction Formula (\cite{Fujino:Subadjunction}), for such a component $W$ of $Z$, there exists an effective $\nQ$-divisor $D_W$ such that $(W,D_W)$ is Kawamata log terminal. This shows that  $Z$ is very close to being MJ-canonical. Along this line, we raise a conjecture in Section 4.
The study of the singularities of $Z$ plays the central role in many applications of linkage theory by the induction method. A typical application can be found in \cite{CU:Reg} on bounding the Castelnuovo-Mumford regularity. 

All above results further lead us to the following theorem about minimal log discrepancies under generic linkage (Theorem \ref{p:42} and Corollary \ref{p:43}). It shows, roughly speaking, the singularities get improved under generic linkage. 
\begin{theorem}	Let $X$ be a codimension $c$ subvariety of a nonsingular variety $A$ and let $Y$ be a generic link of $X$ via $V$. Let $Z=X\cap Y$ be the intersection. Then 
	$$\mld(Z;A, cV)\leq \mld(Z;A,cX)\leq \mld(Z;A,cY).$$
\end{theorem}
 It is worth pointing out that in contrast to the behavior of log canonical thresholds proved in \cite{Niu:SingLink} that $$\lct(A,Y)\geq \lct(A,V)=\lct(A,X),$$
a similar equality $\mld(Z;A, cV)=\mld(Z;A,cX)$ does not hold in general, even when $X$ is nonsingular. Therefore, taking $V$ as a bridge, as we did for lct,  to compare $\mld(Z;A,cY)$ with $\mld(Z;A,cX)$ does not work here.  This issue is discussed in Section 4 and showed by Example \ref{ex:01}.

\vspace{0.3cm}
\noindent{\em Acknowledgement}. We are grateful to Lawrence Ein, Shihoko Ishii, and Bernd Ulrich for valuable discussions and the referee for the suggestions which improve the paper. The author also thanks Claudia Polini for the opportunity to visit the University of Notre Dame to finish this paper. This work was partially supported by AMS-Simons Travel Grants.

\section{Mather-Jacobian singularities}

\noindent Throughout this paper, we work over an algebraically closed field $k$ of characteristic zero. A {\em variety} is a separated reduced and irreducible scheme of finite type over $k$. In this section, we briefly go through the theory of Mather-Jacobian singularities. For further details, we refer the reader to the work \cite{Ishii:MatherDis}, \cite{Roi:JDiscrepancy}, \cite{Ein:MultIdeaMatherDis} and \cite{Ein:SingMJDiscrapency}.

\begin{definition} Let $X$ be a variety of dimension $n$ and let $f:X'\longrightarrow X$ be a log resolution of the Jacobian ideal $\Jac_X$ of $X$. Then the image of the canonical homomorphism
$$f^*(\wedge^n\Omega^1_X)\longrightarrow \wedge^n\Omega^1_{X'}$$
is an invertible sheaf of the form $\Jac_f\cdot\wedge^n\Omega^1_X$, where $\Jac_f$ is the relative Jacobian ideal of $f$. The ideal $\Jac_f$ is invertible and defines an effective divisor $\widehat{K}_{X'/X}$ which is called the {\em Mather discrepancy divisor}.
\end{definition}

\begin{remark}
	In the above definition, the fact that the relative Jacobian ideal $\Jac_f$ is invertible follows from \cite{Lipman:JacobianIdeal} (see also \cite[Remark 2.3]{Ein:MultIdeaMatherDis}). 
\end{remark}

\begin{definition} Let $X$ be a variety and $\mf{a}\subseteq \sO_X$ be an ideal and let $t\in \nQ_{\geq 0}$. For a prime divisor $E$ over $X$, we denote by $C_X(E)$ the center of $E$ on $X$.  Consider a log resolution $\varphi:X'\longrightarrow X$ of $\Jac_X\cdot \mf{a}$ such that $E$ appears in $X'$ and $\mf{a}\cdot\sO_{X'}=\sO_{X'}(-Z)$ and $\Jac_X\cdot\sO_{X'}=\sO_{X'}(-J_{X'/X})$ where $Z$ and $J_{X'/X}$ are effective divisors on $X'$. We define the {\em Mather-Jacobian-discrepancy} ({\em MJ-discrepancy} for short) of $E$ to be
$$a_{\text{MJ}}(E;X,\mf{a}^t)=\ord_E(\widehat{K}_{X'/X}-J_{X'/X}-tZ).$$
The number $a_{\text{MJ}}(E;X,\mf{a}^t)+1$ is called the {\em Mather-Jacobian-log discrepancy} ({\em MJ-log discrepancy} for short). It is independent on the choice of the log resolution $\varphi$.
When $X$ is nonsingular, the MJ-discrepancy $\MJ{a}(E;X,\mf{a}^t)$ is the same as the usual discrepancy $a(E;X,\mf{a}^t)$.

Let $W$ be a proper closed subset of $X$ and let $\eta$ be a point of $X$ such that its closure $\overline{\{\eta\}}$ is a proper closed subset of $X$. We define the {\em minimal MJ-log discrepancy} of $(X,\mf{a}^t)$ along $W$ as
$$\MJ{\mld}(W;X,\mf{a}^t)=\inf_{C_X(E)\subseteq W}\{\ \MJ{a}(E;X,\mf{a}^t)+1\ |\ E \mbox{ a prime divisor over $X$} \}$$
and the {\em minimal MJ-log discrepancy} of $(X,\mf{a}^t)$ at $\eta$ as
$$\MJ{\mld}(\eta;X,\mf{a}^t)=\inf_{C_X(E)=\overline{\{\eta\}}}\{\ \MJ{a}(E;X,\mf{a}^t)+1\ |\ E \mbox{ a prime divisor over $X$} \}.$$
We use the convention that if $\dim X=1$ and the above values are negative, then we set them as $-\infty$. Notice that if $X$ is nonsingular, then minimal MJ-log discrepancy is just the usual minimal log discrepancy and we use the notation $\mld$ for this case.
\end{definition}

Recall that a prime divisor $E$ over a variety $X$ is called {\em exceptional} if there exists a birational morphism $\varphi: Y\longrightarrow X$ such that $Y$ is normal, $E$ is a divisor on $Y$, and $\varphi$ is not isomorphic at the generic point of $E$. Having the definition of MJ-discrepancy as above, we now are able to define MJ-singularities.

\begin{definition}\label{def:51} Let $X$ be a variety. We say that $X$ is {\em MJ-canonical} (resp. {\em MJ-log canonical}) if for every exceptional prime divisor $E$ over $X$, the MJ-discrepancy $\MJ{a}(E;X,\sO_X)\geq 0$ (resp. $\geq -1$) holds.
\end{definition}

\begin{remark}\label{rmk:11} (1) For MJ-log canonical singularities, we have the following equivalent definition: $X$ is MJ-log canonical if and only if the MJ-discrepancy $\MJ{a}(E;X,\sO_X)\geq -1$ for every prime divisor $E$ over $X$ (\cite[Proposition 2.23]{Ein:SingMJDiscrapency}). Simply from the definition of minimal MJ-log discrepancy, we can define MJ-log canonical singularities locally: $X$ is MJ-log canonical if and only if for every closed point $x\in X$, $\MJ{\mld}(x;X,\sO_X)\geq 0$ (\cite[Proposition 2.21]{Ein:SingMJDiscrapency}).

(2) It has been proved in  \cite{Ein:MultIdeaMatherDis} and \cite{Roi:JDiscrepancy} that if a variety has MJ-canonical singularities then it is normal and has rational singularities.

(3) For varieties of dimension one and two, the MJ-log canonical and MJ-canonical singularities have been classified in the work \cite{Ein:SingMJDiscrapency}.
\end{remark}

We shall use the following version of the Inversion of Adjunction. It plays a critical role in transferring singularity information from a variety to its ambient space.  
\begin{theorem}[Inversion of Adjunction, {\cite{Ishii:MatherDis}\cite{Roi:JDiscrepancy}}]\label{p:03}Let $X$ be a codimension $c$ subvariety of a nonsingular variety $A$ defined by the ideal $I_X$.
\begin{itemize}
\item [(1)] Let $W\subset X$ be a proper closed subset of $X$. Then
$$\MJ{\mld}(W;X,\sO_X)=\mld(W;A,I^c_X).$$
\item [(2)] Let $\eta\in X$ be a point such that its closure $\overline{\{\eta\}}$ is a proper closed subset of $X$. Then
$$\MJ{\mld}(\eta;X,\sO_X)=\mld(\eta ;A,I^c_X).$$
\end{itemize}
\end{theorem}
\begin{proof} (1) is a simple version of \cite[Theorem 3.10]{Ishii:MatherDis}. For (2), write $W=\overline{\{\eta\}}$. We can find a small open set $U$ of $\eta$ in $A$ such that $\mld(\eta; A,I^c_X)=\mld(W\cap U ; U, I^c_X|_U)$ and $\MJ{\mld}(\eta ; X,\sO_X)=\MJ{\mld}(W\cap U; X\cap U,\sO_X|_{X\cap U})$. Such open set $U$ can be constructed as follows. Take a log resolution $f:A'\longrightarrow A$ of $I_X\cdot I_W$. Then remove the center $f(E)$ from $A$ for any prime divisor $E
	\subset A'$ such that $f(E)$ does not contain $W$ (hence also remove $f(E)$ if it is a proper subset of $W$). On this open set $U$  we apply the result (1) to get the desired result (2).
\end{proof}

Using the Inversion of Adjunction, we see that MJ-log canonical singularities are essentially the same as log canonical singularities of pairs once we embed the variety in a nonsingular ambient space. We state this observation in the following proposition, which is known to the experts.
\begin{proposition}\label{p:51} Let $X$ be a codimension $c$ subvariety of a nonsingular variety $A$ defined by the ideal $I_X$. Then $X$ is MJ-log canonical if and only if the pair $(A,I^c_X)$ is log canonical.
\end{proposition}
\begin{proof} $X$ is MJ-log canonical if and only if for any closed point $x\in X$, $\MJ{\mld}(x;X,\sO_X)\geq 0$. By Theorem \ref{p:03}, we deduce that $\mld(x;A,I^c_X)=\MJ{\mld}(x;X,\sO_X)\geq0$, which implies that the pair $(A,I^c_X)$ is log canonical.
\end{proof}

The following proposition was implicitly proved in \cite[Proposition 3.22]{Ein:SingMJDiscrapency}. It turns out to be very useful in our study, so we include its proof here. Recall that a closed subset $W$ of a nonsingular variety $A$ is called a {\em log canonical center} for a pair $(A, \mf{a}^t)$ if there is a prime divisor $E$ over $A$ such that $\ord_E(K_{\_/A})-t\ord_E\mf{a}+1\leq 0$ and the center $C_A(E)=W$.
\begin{proposition}\label{p:02} Let $X$ be a codimension $c$ subvariety of a nonsingular variety $A$. Assume that $X$ is MJ-canonical. Then $X$ is the unique log canonical center of the pair $(A,I_X^c)$.
\end{proposition}
\begin{proof} Let $\eta$ be a point of $A$ such that the closure $\overline{\{\eta\}}$ is a proper closed subset of $A$. If $\eta$ is not in $X$, then $\mld(\eta;A,I^c_X)=\mld(\eta;A,\sO_A)\geq 1$. If $\eta$ is in $X$ and $\overline{\{\eta\}}$ is a proper closed subset of $X$, then by the Inversion of Adjunction Theorem \ref{p:03}, $\mld(\eta;A,I^c_X)=\MJ{\mld}(\eta;X,\sO_X)\geq 1$ since $X$ is MJ-canonical (\cite[Proposition 2.21(ii)]{Ein:SingMJDiscrapency}). If $\eta$ is the generic point of $X$, we can compute directly that $\mld(\eta;A,I_X^c)=0$ by blowing up $A$ along $X$.
\end{proof}

The following corollary gives an interesting property that the ideal of a codimension one subvariety of a MJ-canonical variety can be represented as a multiplier ideal.  Recall that given a pair $(A, \mf{a})$, we denote by  $\sI(A,\mf{a}^t)$ the multiplier ideal of weight $t$ associated to the pair. If $\mf{a}$ defines a subscheme $Z$ of $A$, we also use $\sI(A, tZ)$ instead of $\sI(A,\mf{a}^t)$.
\begin{corollary}\label{p:23} Let $X$ be a codimension $c$ subvariety of a nonsingular variety $A$. Let $Z$ be a codimension one subvariety of $X$. Assume that $X$ is MJ-canonical. Then the pair $(A, cZ)$ is canonical and $I_Z=\sI(A,(c+1)Z)$.	
\end{corollary}
\begin{proof} Since $X$ is the only log canonical center of $(A,cX)$ and $Z$ is generically nonsingular, the pair $(A,cZ)$ has no log canonical centers. Hence the pair $(A,cZ)$ is canonical.  So the multiplier ideal $\sI(A,cZ)=\sO_A$. Now by  \cite[Ein's Lemma]{Niu:SingLink}, we see that $I_Z=\sI(A,(c+1)Z)$.
	
\end{proof}

\section{Generic linkage of affine varieties}

\noindent In this section, we study MJ-singularities under generic linkage. A generic link is constructed through a ring extension by adjoining variables. The theory based on this construction was developed by Huneke and Ulrich in the last thirty years and has reached fruitful results. For detailed information of generic linkage, we refer to \cite{HunekeUlrich:DivClass}, \cite{HunekeUlrich:SturLinkage}, \cite{HU88}, and \cite{HunekeUlrich:AlgLinkage}. The paper \cite{Niu:SingLink} also contains the most useful details and backgrounds related to this paper.

\begin{definition}\label{def:01} Let $X_k$ be a codimension $c$ subvariety of a nonsingular affine variety $A_k=\Spec R_k$. A generic link of $X_k$ is defined as follows. Fix a generating set $(f_1,\cdots,f_t)$ of the defining ideal $I_{X_k}$ of $X_k$. Let $(U_{ij}), 1\leq i\leq c, 1\leq j\leq t$, be a $c\times t$ matrix of variables. Set $R=R_k[U_{ij}]$ and $I_X=I_{X_k}R$ and define $A=\spec R$ and $X=\Spec R/I_X$. Notice that $I_X$ is generated by $(f_1,\cdots,f_t)$ in $R$. We define a complete intersection $V$ in $A$ by the ideal
$$I_V=(\alpha_1,\cdots,\alpha_c)=(U_{i,j})\cdot(f_1,\cdots,f_t)^T,$$
that is
$$\alpha_i=U_{i,1}f_1+U_{i,2}f_2+\cdots+U_{i,t}f_t,\quad\quad\quad\mbox{for } 1\leq i\leq c.$$
Then a {\em generic link} of $X_k$ via $V$ is a subscheme $Y$ of $A$ defined by the ideal $I_Y=(I_V:I_X)$.
\end{definition}

\begin{remark}
	(1) The generic link $Y$ is a subvariety (reduced and irreducible) of $A$ \cite[Proposition 2.6]{HunekeUlrich:DivClass}. Furthermore, $X$ and $Y$ are actually geometrically linked, i.e., the complete intersection $V$ has only $X$ and $Y$ as its irreducible components. 
	
	(2) Clearly, the natural morphism $X\longrightarrow X_k$ is smooth. Many singularities are preserved under smooth morphisms, such as rational singularities and log canonical singularities of pairs. MJ-singularities are also preserved by \cite[Corollary 2.8]{Niu:VanishingThem}. 
\end{remark}

We start with proving the MJ-log canonical case, which is a direct application of the Inversion of Adjunction plus the results of \cite{Niu:SingLink}.
\begin{theorem}\label{thm:31}Let $X_k$ be a subvariety of a nonsingular affine variety $A_k$. Let $Y$ be a generic link of $X_k$. Assume that $X_k$ has MJ-log canonical singularities. Then $Y$ also has MJ-log canonical singularities.
\end{theorem}
\begin{proof}
Applying Proposition \ref{p:51}, we see that the pair $(A_k, I^c_{X_k})$ is log canonical. Then apply \cite[Corollary 3.10]{Niu:SingLink}, we deduce that the pair $(A,I^c_Y)$ is log canonical. Finally, using Proposition \ref{p:51} again, we conclude that $Y$ is MJ-log canonical. 
\end{proof}

In the rest of this section, we focus ourselves on the case of MJ-canonical singularities. The essential point in our approach is to analyze the intersection divisor $Z=X\cap Y$. Since $X$ and $Y$ are linked by a complete intersection $V$, it is easy to see that $Z$ is purely codimension one in $X$ and $Y$.

\begin{theorem}\label{thm:21}Let $X_k$ be a subvariety of a nonsingular affine variety $A_k$. Let $Y$ be a generic link of $X_k$. Assume that $X_k$ has MJ-canonical singularities. Then
	\begin{itemize}
		\item [(1)] $Y$ has MJ-canonical singularities;
		\item [(2)] The intersection $Z=X\cap Y$ has disjoint irreducible components and each of them is a normal subvariety of codimension one in $X$. Furthermore, let $W$ be an irreducible component of $Z$. Then there exists an effective $\nQ$-divisor $D_W$ on $W$ such that $(W, D_W)$ is Kawamata log terminal (klt).
	\end{itemize}
\end{theorem}
\begin{proof} Let $\varphi_k: \overline{A}_k\longrightarrow A_k$ be a factorizing resolution of singularities (for the definition, see for example \cite[Definition 2.6, Remark 2.7]{Niu:SingLink}) of $X_k$ inside $A_k$, so that $$I_{X_k}\cdot\sO_{\overline{A}_k}=I_{\overline{X}_k}\cdot\sO_{\overline{A}_k}(-G_k),$$ where $\overline{X}_k$ is the strict transform of $X_k$, $G_k$ is an effective divisor supported on the exceptional locus of $\varphi_k$, and furthermore $\overline{X}_k$ and the exceptional locus of $\varphi_k$ are simple normal crossings. The morphism $\varphi_k$ can be assumed to be an isomorphism over the open set $A_k\backslash X_k$. 
Next, we blow up $\overline{A}_k$ along $\overline{X}_k$ to get
$$\mu_k:\widetilde{A}_k=\bl_{\overline{X}_k}\overline{A}_k\longrightarrow \overline{A}_k$$
such that $I_{\overline{X}_k}\cdot \sO_{\widetilde{A}_k}=\sO_{\widetilde{A}_k}(-T_k)$, where $T_k$ is an exceptional divisor of $\mu$ and it is a prime divisor since $\overline{X}_k$ is nonsingular. Denote the composition $\phi_k\circ\mu_k$ by $$\psi_k=(\phi_k\circ\mu_k): \widetilde{A}_k\stackrel{\mu_k}{\longrightarrow} \overline{A}_k\stackrel{\varphi_k}{\longrightarrow} A_k.$$
Notice that $\psi_k$ is a log resolution of the pair $(A_k,I_{X_k})$ and $I_{X_k}\cdot\widetilde{A}_k=\sO_{\widetilde{A}_k}(-T_k-\mu^*_kG_k)$.

Now we tensor $k[U_{ij}]$ to the construction above and make the base change for all the objects involved, including morphisms, varieties and divisors. We keep using the same letters without the subscript $k$ to represent the corresponding objects after base change. For instance, $\overline{A}=\overline{A}_k\otimes_k \Spec k[U_{ij}]$, $\overline{X}=\overline{X}_k\otimes_k\Spec k[U_{ij}]$, $G=G_k\otimes_k \Spec k[U_{ij}]$, $T=T_k\otimes_k \Spec k[U_{ij}]$ and etc. Therefore, we obtain a factorizing resolution of singularities of $X$ inside $A$ as
$$\varphi:\overline{A}\longrightarrow A,$$
such that $$I_X\cdot\sO_{\overline{A}}=I_{ \overline{X}}\cdot\sO_{ \overline{A}}(-G),$$ where the nonsingular variety $\overline{X}$ is the strict transform of $X$, $G$ is an effective divisor supported on the exceptional locus of $\varphi$, and $\overline{X}$ and exceptional locus of $\varphi$ are simple normal crossings. The composition
$$\psi:\widetilde{A}\stackrel{\mu}{\longrightarrow} \overline{A}\stackrel{\varphi}{\longrightarrow} A$$
is a log resolution of the pair $(A,I_X)$, where $\mu:\widetilde{A}\longrightarrow \overline{A}$ is the blowup of $\overline{A}$ along $\overline{X}$ with an exceptional divisor $T$, such that $$I_{X}\cdot\widetilde{A}=\sO_{\widetilde{A}}(-T-\mu^*G).$$

We make the following claim. Note that partial results of the claim has been proved in details in \cite[Claim 3.1.1, Claim 3.1.2]{Niu:SingLink}. However, in order to keep the consistence of notations, we still include them here.
\begin{claim}\label{clm:01} We have the following statements.

\begin{itemize}
\item [(1)] There is a decomposition
\begin{equation}\label{eq:03}
I_V\cdot\sO_{\overline{A}}=I_{\overline{V}}\cdot\sO_{\overline{A}}(-G),
\end{equation}
where $I_{\overline{V}}$ is an ideal sheaf on $\overline{A}$ and defines a complete intersection $\overline{V}$ of $\overline{A}$.
\item [(2)] Write $\overline{Y}$ to be the strict transform of $Y$ under the birational morphism $\varphi$. Then $\overline{Y}$ is linked to $\overline{X}$ via $\overline{V}$. More precisely, there is an affine open cover of $\overline{A}_k$ such that on each open set of the cover, $\overline{Y}$ is a generic link of $\overline{X}_k$ via $\overline{V}$ in the sense of Definition \ref{def:01}.
\item [(3)] There is a decomposition
\begin{equation}\label{eq:04}
I_{\overline{V}}\cdot\sO_{\widetilde{A}}=I_{\widetilde{Y}}\cdot \sO_{\widetilde{A}}(-T),
\end{equation}
where $I_{\widetilde{Y}}$ is an ideal on $\widetilde{A}$ and defines a nonsingular variety $\widetilde{Y}$ as the strict transform of $Y$ under the birational morphism $\psi$.
\item [(4)] The nonsingular variety $\widetilde{Y}$ meets the strata of the divisors supported in $\exc(\psi)$ as normal crossings. Precisely, for arbitrary $r(\geq 0)$ prime divisors $E_1,\cdots,E_r$ supported in  $\exc(\psi)$, the intersection $\widetilde{Y}\cap E_1\cap\cdots\cap E_r$ is either empty or a nonsingular subscheme of the expected dimension. In particular, $\widetilde{Y}\cap T$ is a nonsingular divisor in $\widetilde{Y}$.
\end{itemize}
\end{claim}
\textit{Proof of Claim \ref{clm:01}.} The question is local. Let $\overline{U}_k=\Spec \overline{R}_k$ be an affine open set of $\overline{A}_k$ such that the effective divisor $G_k$ is defined by an equation $g\in \overline{R}_k$. By the construction that $\overline{A}_k$ is a factoring resolution of $X_k$ in $A_k$, we have a decomposition $I_{X_k}\cdot\overline{R}_k=I_{\overline{X}_k}\cdot (g)$ on $\overline{U}_k$. Notice that $I_{X_k}\cdot \overline{R}_k=(f_1,\cdots,f_t)\cdot \overline{R}_k$ so we can write $f_i=\overline{f}_ig$ for some $\overline{f}_i\in \overline{R}_k$ for $i=1,\cdots,t$ and therefore $I_{\overline{X}_k}=(\overline{f}_1,\cdots,\overline{f}_t)$. By base change, we set $\overline{R}=\overline{R}_k[U_{i,j}]$ so that  $\overline{U}=\Spec \overline{R}$ is an affine open set of $\overline{A}$. Notice that on $\overline{U}$ the ideal $I_{\overline{X}}=I_{\overline{X}_k}\cdot \overline{R}$ and the effective divisor $G$ is still defined by the equation $g$. Recall that the ideal $I_V=(\alpha_1,\cdots,\alpha_c)$, where $\alpha_i=U_{i,1}f_1+U_{i,2}f_2+\cdots+U_{i,t}f_t$. Thus if we write $$\overline{\alpha}_i=U_{i,1}\overline{f}_1+U_{i,2}\overline{f}_2+\cdots+U_{i,t}\overline{f}_t, \mbox{ for }i=1,\cdots, c$$ and set $I_{\overline{V}}=(\overline{\alpha}_1,\cdots,\overline{\alpha}_c)$, then $I_{\overline{V}}$ is a complete intersection on $\overline{U}$. It is then clear that we have the desired decomposition $I_{V}\cdot \overline{R}=I_{\overline{V}}\cdot (g)$ on $\overline{U}$ as in the statement (1). Notice that $\varphi$ is an isomorphism outside $X$ and then (2) is a directly consequence of the local construction above.

For the statement (3), we continue to work locally. We blow up $\overline{U}_k$ along $I_{{\overline{X}}_k}$ and then use base change to obtain the blowup $\overline{U}$ along the ideal $I_{\overline{X}}$. Take a canonical affine cover of the blowup $\Bl_{\overline{X}_k}\overline{U}_k$ and then we proceed on each open set of this cover. Without loss of generality, we set
\begin{equation}\label{eq:01}
S=\overline{R}_k[\overline{f}_2/\overline{f}_1,\cdots, \overline{f}_t/\overline{f}_1],
\end{equation}
and then $\widetilde{U}_k=\Spec S$ is an open set of the cover such that the exceptional divisor $T$ is given by the element $\overline{f}_1$ on $\widetilde{U}_k$. Accordingly, $\widetilde{U}=\Spec S[U_{i,j}]$ is an open set of $\Bl_{\overline{X}}\overline{U}$. Write
\begin{equation}\label{eq:02}
\widetilde{\alpha}_i=U_{i,1}+U_{i,2}\overline{f}_2/\overline{f}_1+\cdots+U_{i,t}\overline{f}_t/\overline{f}_1, \mbox{ for }i=1,\cdots,c,
\end{equation}and set $I_{\widetilde{Y}}=(\widetilde{\alpha}_1,\cdots,\widetilde{\alpha}_c)$. Then on the open set $\widetilde{U}$ we have $I_{\overline{V}}\cdot \sO_{\widetilde{U}}=I_{\widetilde{Y}}\cdot (f_1)$ and $I_{\widetilde{Y}}$ defines an irreducible nonsingular variety $\widetilde{Y}$ on $\widetilde{U}$ since
$$\widetilde{Y}=\Spec S[U_{i,j}]/(\widetilde{\alpha}_1,\cdots,\widetilde{\alpha}_c)$$
and each $\widetilde{\alpha}_i$ is essentially defined by a variable. Since $\psi$ is an isomorphism outside $X$, $\widetilde{Y}$ is clearly the strict transform of $Y$ by our construction.

For the statement (4), by the construction, $\exc(\varphi)$ is the base change of $\exc(\varphi_k)$. Hence each $E_i$ is the base change of an exceptional divisor $E_{i,k}$ on $\widetilde{A}_k$. Consider an irreducible component $B$ of the intersection $E_1\cap\cdots\cap E_r$. It is clear that $B$ can be obtained by base change of a corresponding irreducible component $B_k$ of the intersection  $E_{1,k}\cap\cdots\cap E_{r,k}$. Hence, we can assume that $B$ is defined by
$$I_{B}=I_{B_k}[U_{i,j}].$$
Then the structure sheaf of $\widetilde{Y}\cap B$ is
$$S[U_{i,j}]/(I_{\widetilde{Y}}+I_B)=\sO_{B_k}[U_{i,j}]/(\widetilde{\alpha}'_1,\cdots,\widetilde{\alpha}'_c),$$
where $\widetilde{\alpha}'_i=U_{i,1}+U_{i,2}\widetilde{s}_2+\cdots+U_{i,t}\widetilde{s}_t$ and $\widetilde{s}_j$ is the image of $\overline{f}_j/\overline{f}_1$ in the ring $\sO_{B_k}$. Because $B_k$ is a regular subscheme of expected dimension and $\widetilde{\alpha}'_j$'s are all variables, this local computation shows that $\widetilde{Y}\cap B$ is a nonsingular subscheme of the expected dimension. This finishes the proof of Claim \ref{clm:01}.\\

We blow up $\tilde{A}$ along $\tilde{Y}$ to get $\nu: A'=\bl_{\tilde{Y}}\tilde{A}\longrightarrow\tilde{A}$ with an exceptional divisor $E_Y$. The composition $\rho=\nu\circ\psi:A'\longrightarrow A$ is a log resolution of $I_V\cdot I_X$, such that 
$$I_V\cdot\sO_{A'}=\sO_{A'}(-E_Y-E_X-P), \text{ and } I_X\cdot\sO_{A'}=\sO_{A'}(-E_X-P),$$
where $E_X=\nu^*(T)$ and $P=(\nu\circ\mu)^*(G)$ and $E_Y\cup E_X\cup P\cup \exc(\rho)$ has a simple normal crossing support. Since $X_k$ is MJ-canonical, the variety $X$ is also MJ-canonical \cite[Corollary 2.8]{Niu:VanishingThem}. By Proposition \ref{p:02}, $X$ is the unique log canonical center for the pair $(A,I^c_X)$. Furthermore, by the Inversion of Adjunction, the pair $(A,I_X^c)$ is log canonical and therefore $(A,I_V^c)$ is log canonical (\cite{Niu:SingLink}). A direct computation shows that the discrepancy $$a(E_Y; A, I_V^c)=a(E_X;A,I_V^c)=-1.$$  
For any other exceptional divisor $E$ on $A'$, the discrepancy $a(E; A, I_V^c)\geq 0$. Thus if $F$ is a prime divisor over $A$ with $a(F;A,I_V^c)=-1$, then we have either $F\in \{E_X,E_Y\}$ or the center of $F$ on $A'$ is an irreducible component of $E_X\cap E_Y$.

Denote by $I_Z=I_X+I_Y$ the ideal of $Z=X\cap Y$. Since both $Y$ and $X$ are log canonical centers for $(A,I^c_V)$, every irreducible component of $Z$ is also a log canonical center for $(A,I^c_V)$. We take a log resolution of $I_V\cdot I_X\cdot I_Z$ as $f:A''\longrightarrow A$ satisfying the following conditions.  
\begin{itemize}
	\item [(i)] It factors through $\rho$, i.e,
	$$\xymatrix{
		A'' \ar[rr]\ar[dr]_{f} &   & A' \ar[dl]^{\rho} \\
		& A &
	}$$
	\item [(ii)] For each irreducible component $W$ of $Z$, there exists a prime divisor $F\subset A''$ such that $a(F;A,I^c_V)=-1$ and $C_A(F)=W$.
\end{itemize}
We choose a number $a\in\nQ$ satisfying the condition that if $F\subset A''$ is a prime divisor such that $a(F;A,I^c_V)=-1$ and $C_A(F)\subseteq Z$, then we have the inequality $$\ord_FI_V<a\ord_FI_Z.$$
After fixing such number $a$, we further choose a rational number $0<\epsilon \ll 1$ such that for all prime divisor $F$ on $A''$ with $a(F;A,I^c_V)>-1$ and $C_A(F)\subseteq Z$, we have
$$a(F;A,I^c_V)+\epsilon (\ord_FI_V-a\ord_FI_Z)>-1.$$
Notice that by the construction, for any prime divisor $F\subset A''$, the discrepancy $$a(F; A,(c-\epsilon)V+ {\epsilon}aZ)=a(F;A,I^c_V)+\epsilon (\ord_FI_V-a\ord_FI_Z).$$ Now we observe that if $F\subset A''$ is a prime divisor, then  $a(F; A,(c-\epsilon)V+ {\epsilon}aZ)\leq -1$ if and only if $a(F;A,I^c_V)\leq -1$ and $C_A(F)\subseteq Z$. Furthermore, its center $C_{A'}(F)$ must be an irreducible component of $E_X\cap E_Y$. Applying the connectedness theorem to the pair $(X, (c-\epsilon)V+ {\epsilon}aZ)$ yields that the induced map $$\rho:E_X\cap E_Y\longrightarrow Z$$ is dominant and has connected fibers. But $E_X\cap E_Y$ is nonsingular so that its irreducible components are all disjoint. This implies that each component of $Z$ must be dominated by only one component of $E_X\cap E_Y$ and the components of $Z$ are disjoint. Therefore, $X$, $Y$ and the irreducible components of $Z$ are the only log canonical centers of the pair $(A, I^c_V)$. It follows that $Y$ and the irreducible components of $Z$ are the only possible log canonical centers of the pair $(A, I^c_Y)$.

Now since $X$ is MJ-canonical, it has rational singularities. By \cite{Niu:SingLink}, $Y$ is normal and therefore is nonsingular at the generic points of $Z$. Thus any irreducible component of $Z$ cannot be a log canonical center of the pair $(A,I^c_Y)$. Hence $Y$ is the only log canonical center of $(A,I^c_Y)$. Finally, by the Inversion of Adjunction, we deduce that $Y$ is MJ-canonical. 

For the statement (2), we take two general effective divisor $D_1$ and $D_2$ such that $I_{D_1}\subset I^2_V$ and $I_{D_2}\subset I^2_Z$. Set $D'_1=\frac{1}{2}D_1$ and $D'_2=\frac{1}{2}D_2$. Since each irreducible component of $Z$ is a minimal log canonical center of $(A, (c-\epsilon)V+ {\epsilon}aZ)$, by the general choice of $D'_1$ and $D'_2$, we see that each irreducible component of $Z$ is also a minimal log canonical center of $(A, (c-\epsilon)D'_1+ {\epsilon}aD'_2)$. The result then follows from Local Subadjunction Formula in \cite[Theorem 7.2]{Fujino:Subadjunction}.
\end{proof}

\begin{remark} In the proof above, we show that every irreducible component of the intersection $Z$ is a minimal log canonical center of $(A, cV)$, provided that $X$ is MJ-canonical. Unfortunately, our proof cannot show $Z$ itself is irreducible, which seems very natural. To get a better understanding of this issue, we quote the following result of Johnson-Ulrich about the intersection divisor $Z$.
\end{remark}

\begin{proposition}[\cite{Johnson:SerreGeoLink}]\label{p:06} Let $X_k$ be a subvariety of a nonsingular affine variety $A_k$. Let $Y$ be a generic link of $X_k$ and let $Z=X\cap Y$ be the intersection defined by $I_X+I_Y$. The following are equivalent.
\begin{itemize}
\item [(1)] $Z$ is integral.
\item [(2)] $Z$ is reduced.
\item [(3)] $X$ is a complete intersection locally in codimension one.
\end{itemize}	
\end{proposition}

Now from this proposition, the intersection divisor $Z$ in Theorem \ref{thm:21} is irreducible. So we conclude the following corollary.
\begin{corollary}\label{p:41}Let $X_k$ be a subvariety of a nonsingular affine variety $A_k$ such that $X_k$ has MJ-canonical singularities. Let $Y$ be a generic link of $X_k$. Then the intersection $Z=X\cap Y$ defined by the ideal $I_X+I_Y$ is an irreducible codimension one subvariety of $X$ and there exists an effective $\nQ$-divisor $D_Z$ on $Z$ such that the pair $(Z,D_Z)$ is klt.
\end{corollary}

\begin{remark} In many applications, one starts with a variety $X$ of codimension $c$ in a nonsingular ambient space $A$ and  constructs a sequence of varieties $$X=X_0\sim X_1 \sim X_2\sim\cdots,$$ in which $X_{i+1}$ is a generic link of $X_{i}$. Two varieties connected by such a sequence are called in the same linkage class. Hence properties that preserved under linkage will be particularly interesting. For singularities, the situation is complicated. For instance, rational singularities, which in some sense are simple, cannot be preserved under linkage. The conditions imposed to preserve rational singularities are fairly strong \cite{Niu:SingLink}. But log canonical singularities of pairs (essentially the same as MJ-log canonical singularities) are preserved \cite{Niu:SingLink}, i.e., if $(A,cX)$ is log canonical, then in the linkage sequence above $(A,cX_i)$ is always log canonical. However, roughly speaking, log canonical singularities are still too broad. The theorem we just proved provides a new type of singularities, MJ-canonical singularities, which are stronger than rational singularities but  can be preserved under linkage. It would be interesting to investigate further linkage classes containing MJ-canonical singularities. 
\end{remark}

\section{Variants and conjectures.}

\noindent Let $X_k$ be a codimension $c$ subvariety in a  nonsingular ambient variety $A_k$. Let $Y$ be a generic link of $X$ via a complete intersection $V$, as in Definition \ref{def:01}. Write $Z=X\cap Y$. It has been showed in \cite{Niu:SingLink} that for log canonical thresholds ($\lct$),
$$\lct(A,X)=\lct(A,V)\leq \lct (A,Y).$$
The way to establish this result is to show first the equality  $\lct(A,X)=\lct(A,V)$ and then the second inequality follows immediately as $I_V\subseteq I_Y$. Turning to another important invariant of singularities, namely minimal log discrepancy, it is expected that one could establish a similar result
$$\mld(Z;A,cX)=\mld(Z;A,cV)\leq \mld(Z;A,cY),$$
from which Theorem \ref{mainthm} would follow immediately by the Inversion of Adjunction. Unfortunately, the situation for $\mld$ is more complicated than $\lct$ and the above ideal equality ``$\mld(Z;A,cX)=\mld(Z;A,cV)$" does not hold, as showed in the following example. 

\begin{example}\label{ex:01}Consider $X_k$ in $k[x, y]$ defined by the equation $x$. We construct a generic
	link of $X_k$ as follows: $X$ is still defined by $x$ in $k[x, y, u]$, the complete intersection
	$V$ is defined by $xu$ and a generic link $Y$ is then defined by $u$. Now $X$ and $Y$ are two
	coordinate planes in $\nA^3$. The intersection $Z = X \cap Y$ is a line defined by the ideal $(x, u)$.
	Then we see
	$$\mld(Z;\nA^3, V ) = 0,\quad \text{and } \mld(Z;\nA^3, X) = 1.$$
Hence the equality $\mld(Z;A,cX)=\mld(Z;A,cV)$ is not true in general.	
\end{example}
 
This example shows that  we cannot directly use the complete intersection $V$ as a bridge to compare the minimal log discrepancies of $X$ and $Y$. However, based on what we have established in the preceding section, we are still able to prove the following theorem which reveals the behavior of minimal log discrepancies under generic linkage.


\begin{theorem}\label{p:42}Let $Y$ be a generic link of a variety $X_k$ in a nonsingular affine space $A_k$. Let $Z=X\cap Y$ and let $c=\codim_A X$. Then one has
	$$\mld(Z;A, cV)\leq \mld(Z;A,cX)\leq \mld(Z;A,cY).$$	
\end{theorem}
\begin{proof} Notice that the minimal log discrepancies involved in the theorem are either nonnegative integers or $-\infty$. The first inequality $\mld(Z;A, cV)\leq \mld(Z;A,cX)$ is obvious because $I_V\subseteq I_X$. In the sequel, we shall prove the second inequality
	\begin{equation}\label{eq:06}
	\mld(Z;A,cX)\leq \mld(Z;A,cY).
	\end{equation}

If $\mld(Z;A,cX)=-\infty$, then the inequality (\ref{eq:06}) is obvious. So we may assume that $\mld(Z;A,cX)\geq 0$.  Thus the pair $(A,cX)$ is log canonical in a neighborhood $U\subset A$  of $Z$. Using the construction in the proof of Theorem \ref{thm:21}, we consider the restriction of the resolutions of singularities $\varphi$ and $\psi$ over the open set $U$. Following the same argument, we can show that the pair $(A,cV)$ is log canonical on $U$ and therefore $(A,cY)$ is log canonical on $U$  (see also \cite{Niu:SingLink} for a proof in this case). Thus $\mld(Z;A,cY)\geq 0$. This means particularly that if $\mld(Z;A,cX)=0$, then the inequality  (\ref{eq:06}) holds.
	
	Next we assume that $\mld(Z;A,cX)\geq 1$. Thus the pair $(A,cX)$ has no log canonical centers contained in $Z$. By removing the log canonical centers that are properly contained in $X$, we may assume that on the open set $U$, the pair $(A,cX)$ is log canonical and has a unique log canonical center $X$. Following the proof of Theorem \ref{thm:21} again by restricting $\varphi$ and $\psi$ over the open set $U$, we conclude that on the open set $U$, $(A,cY)$ has a unique log canonical center $Y$ and the intersection $Z$ is normal and of codimension one in $X$. By the Inversion of Adjunction, we deduce that on the open set $U$, both $X$ and $Y$ are MJ-canonical and therefore are normal and nonsingular at the generic point of $Z$. Consequently, we obtain $\mld(Z,A,cX)=\mld(Z,A,cY)=1$. Hence the inequality (\ref{eq:06}) still holds.
	
\end{proof}

Using the Inversion of Adjunction, we immediately have the following corollary concerning minimal MJ-log discrepancies under linkage.
\begin{corollary} Let $Y$ be a generic link of a variety $X_k$ in a nonsingular affine space $A_k$ and let $Z=X\cap Y$. Then one has
		$$\MJ{\mld}(Z;X,\sO_X)\leq \MJ{\mld}(Z;Y,\sO_Y).$$	
\end{corollary}
The intersection divisor $Z$ contains the singular locus of $Y$ and therefore governs the singularities of $Y$. Besides, it also plays central role when we  try to use the induction method for a generic link. For instance, in \cite{CU:Reg}, Chardin-Ulrich studied the singularities of $Z$ and then proceeded by induction on dimensions to obtain a bound for Castelnuovo-Mumford regularity. Motivated by Theorem \ref{thm:21} (2), we propose the following conjecture.

\begin{conjecture}Let $Y$ be a generic link of a variety $X$ in a nonsingular affine space $A$. Assume that $X$ is MJ-canonical. Then the intersection $Z=X\cap Y$ is also MJ-canonical. 
\end{conjecture}

\begin{remark} Under the assumption of the conjecture, it seems reasonable to show first a weaker result that the pair $(A, (c+1)Z)$ is log canonical, where $c=\codim_AX$. Note that  by Corollary \ref{p:23}, the pair $(A,cZ)$ is already canonical and $I_Z=\sI(A,(c+1)Z)$, which is very close to the pair $(A, (c+1)Z)$ being log canonical.
	
\end{remark}

Finally, we discuss variant settings of generic linkage in application. The results we have obtained can be easily established for these settings, so we leave the details to the reader. 

\begin{definition}\label{def:22}	Let $A=\Spec R$ be an affine nonsingular variety and let $X\subset A$ be a subvariety of codimension $c$. Fix a set of generators for the ideal $I_X$ as $f_1,\cdots, f_t$. Define a complete intersection $V$  by the equations $$\alpha_i=a_{i,1}f_1+\cdots+a_{i,t}f_t,$$ for $i=1,\cdots, c$, where $a_{i,j}$'s are general scalars in $k$. Then a {\em general link} $Y$ (or generic link, to be consistent with Definition \ref{def:01}) is defined by the ideal $I_Y=(I_V:I_X)$. 
\end{definition}

\begin{definition}\label{def:21}Let $X$ be a codimension $c$ subvariety in a  nonsingular  variety $A$. Let $L$ be a line bundle on $A$ such that $X$ is cut out by $t\ (\geq c)$ sections $$f_1,\cdots, f_t \in H^0(A,L).$$
	Choose general $c$ sections from the linear space $\langle f_1,\cdots, f_t\rangle\subseteq H^0(A,L)$ such that they cut out a complete intersection $V$ of $A$. A {\em general link} (or generic link, to be consistent with Definition \ref{def:01}) $Y$ of $X$ is defined by the ideal $I_Y=(I_V:I_X)$.	
\end{definition}

\begin{remark}\label{rmk:31} (1). Strictly speaking, the meaning of ``choose general scalars or general sections" in the above definitions should depend on certain properties that we want to show for a general link. When we prove results for these settings of general links, those properties (such as singularities) should be clear from the context.
	
	(2). Unlike in the algebraic setting of Definition \ref{def:01}, the generic link $Y$ in Definition \ref{def:22} and \ref{def:21} could be empty. This happens if and only if the variety $X$ itself is already a complete intersection in $A$. 

	(3). If $X$ is a projective variety in a projective space $A=\nP^N$ defined by an ideal sheaf $\sI_X$, then we can take a number $d$ such that $\sI_X(d)$ is globally generated. In this case, take $L=\sO_{\nP^N}(d)$ and the complete intersection $V$ is also cut out by degree $d$ equations.
	
	(4). If we cover $X$ by affine open sets, we can easily reduce the case of Definition \ref{def:21} to the case of Definition \ref{def:22}. A general link in Definition \ref{def:22} is just a general fiber over $\Spec k[U_{i,j}]$ of the generic link $Y$ in Definition \ref{def:01}. Hence our main result can be established for all of those settings, which we state as the following corollaries. 
\end{remark}

\begin{corollary}Let $X$ be a subvariety of a nonsingular variety $A$ and let $Y$ be a generic link of $X$ (in the sense of Definition \ref{def:22} and \ref{def:21}). 
	
	\begin{itemize}
		\item [(1)] If $X$ is MJ-canonical (resp. MJ-log canonical), then so is $Y$.
		\item [(2)] If $X$ is MJ-canonical, then the irreducible components of the intersection $Z=X\cap Y$ are disjoint and of codimension one in $X$. Furthermore, for such a component $W$ of $Z$, there exists an effective $\nQ$-divisor $D_W$ on $W$ such that the pair $(W,D_W)$ is klt.
	\end{itemize}
\end{corollary}

\begin{corollary}\label{p:43} Let $X$ be a codimension $c$ subvariety of a nonsingular variety $A$ and let $Y$ be a generic link of $X$ (in the sense of Definition \ref{def:22} and \ref{def:21}). Let $Z=X\cap Y$ be the intersection. Then 
	$$\mld(Z;A, cV)\leq \mld(Z;A,cX)\leq \mld(Z;A,cY).$$
\end{corollary}

\begin{remark} Alternatively, one may prove above corollaries by constructing appropriate resolutions of singularities for $X$ and $V$, paralleling to the proof of Theorem \ref{thm:21}. We outline this approach here for the convenience of the reader. Take a factorizing resolution of singularities of $X$ inside $A$ as $\varphi:\overline{A}\longrightarrow A$ such that $I_X\cdot \sO_{\overline{A}}=I_{\overline{X}}\cdot\sO_{\overline{A}}(-G)$ where the nonsingular variety $\overline{X}$ is the strict transform of $X$, $G$ is an effective divisor supported on  $\exc(\varphi)$, and $\overline{X}$ and $\exc(\varphi)$ are simple normal crossings. Blow up $\overline{A}$ along $\overline{X}$ to get $\mu:\widetilde{A}=\Bl_{\overline{X}}\overline{A}\longrightarrow \overline{A}$  with an exceptional divisor $T$. Then the composition $\psi=\varphi\circ\mu:\widetilde{A}\longrightarrow A$	is a log resolution of $I_X$ satisfying the condition that $I_{X}\cdot\widetilde{A}=\sO_{\widetilde{A}}(-T-\mu^*G)$. Since the ideal $I_V$ is generated by $c$ general equations in $I_X$, using Bertini's theorem we get that $I_{V}\cdot\sO_{\widetilde{A}}=I_{\widetilde{Y}}\cdot \sO_{\widetilde{A}}(-T-\mu^*G)$ and $\widetilde{Y}$ is a nonsingular variety on $\widetilde{A}$ resolving the singularities of $Y$. Now this construction satisifies all the properties in Claim \ref{clm:01}. Follow the same argument in the proof of Theorem \ref{thm:21}, the corollaries above can be easily proved.

\end{remark}

\begin{remark} Since we do not have a similar Proposition \ref{p:06} of Johnson-Ulrich for the cases of Definition \ref{def:22} and \ref{def:21}, it is not clear to us that if $Z$ is irreducible in the above corollaries. In addition, the general link $Y$ might not be irreducible either.  	
\end{remark}

\bibliographystyle{alpha}

\end{document}